\documentclass[reqno,12pt]{amsart}        
\usepackage{amsfonts}
\usepackage{colordvi}
\usepackage{amsmath,amssymb,amsthm} 
\usepackage{mathrsfs}
\usepackage{graphicx}
\usepackage{xypic}

\theoremstyle{plain}      
\newtheorem{thm}{Theorem}[section]     
\newtheorem{theorem}[thm]{Theorem}     
\newtheorem{corollary}[thm]{Corollary}     
\newtheorem{lemma}[thm]{Lemma}     
\newtheorem{proposition}[thm]{Proposition}

\theoremstyle{remark}

\theoremstyle{definition}      
\newtheorem{definition}[thm]{Definition}

\def\Jac{{\mathop{\rm Jac}}}

\newcommand{\R}{\mathbb R}

\newcommand{\Arg}{\mathrm{Arg}\,}

\newcommand{\Log}{\mathrm{Log}\,}
\newcommand{\A}{\mathscr A}

\newcommand{\Ree}{\textrm{Re}}
\newcommand{\Ima}{\textrm{Im}}
\newcommand{\Critp}{{\operatorname{Critp}}}

\begin{document}

\title[Amoebas and coamoebas of  linear spaces]
  {Amoebas and coamoebas of linear spaces}

\author{Mounir Nisse and Mikael Passare}
\date{}

\address{School of Mathematics KIAS, 85 Hoegiro Dongdaemun-gu, Seoul 02455, Republic of Korea.}
\email{mounir.nisse@gmail.com}
\thanks{Research of  the first author is partially supported by NSF MCS grant DMS-0915245.}

\subjclass[2010]{14T05, 32A60}
\keywords{Very affine linear spaces, amoebas, coamoebas, logarithmic Gauss map}
\maketitle

\begin{abstract}
We give a complete description of  amoebas and coamoebas of $k$-dimensional very affine linear
spaces in $(\mathbb{C}^*)^{n}$.  This include an  upper bound of their dimension, and we   show  that if a $k$-dimensional very affine linear space in  $(\mathbb{C}^*)^{n}$
 is generic, then the dimension of its (co)amoeba  is equal to $\min \{  2k, n\}$. Moreover, we prove that the volume of its  coamoeba is equal  to  $\pi^{2k}$. In addition,  if the space  is  generic and real, then the volume of its amoeba is equal to $\frac{\pi^{2k}}{2^k}$.
\end{abstract}


\section{Introduction}\label{sec:1}

Amoebas and coamoebas are very fascinating notions in mathematics, the first  has been introduced by I. M. Gelfand,
 M. M. Kapranov and A. V. Zelevinsky
in 1994  \cite{GKZ-94}, and the second  by the second author in a talk in  2004. They  are natural projections 
of complex varieties, and which turn out to have relations to several other fields: tropical geometry, real algebraic geometry, 
generalized hypergeometric functions, mirror symmetry, and others (e.g., \cite{IMS-07}, \cite{MS}, \cite{M1-04}, \cite{M2-04}, \cite{NS1-11}, \cite{PR-04}). More precisely, the amoebas (respectively  coamoebas) of  complex algebraic 
and generally analytic varieties in the complex algebraic torus $(\mathbb{C}^*)^n$ are defined as  their image under the logarithmic mapping 
$\Log :(z_1,\ldots ,z_n)\mapsto (\log |z_1| ,\ldots ,\log |z_n| )$ (respectively the argument mapping 
$\Arg : (z_1,\ldots ,z_n)\mapsto (\frac{z_1}{|z_1|},\ldots ,\frac{z_1}{|z_1|})$).  Amoebas (respectively coamoebas) are the link between  
classical complex algebraic geometry and  tropical (respectively complex tropical) geometry. More precisely,  amoebas  degenerate to
 piecewise-linear objects called tropical varieties (see \cite{M1-04}, and \cite{PR-04}), and comoebas degenerate to a non-Archimedean 
coamoebas which are  the coamoebas of some lifting in the complex algebraic torus of tropical varieties. See \cite{NS2-11} for more 
details about non-Archimedean coamoebas, and \cite{N2-09}  about this degeneration in case of hypersurfaces.
Whereas the theory of (co)amoebas of complex hypersurfaces is by now reasonably well understood (see e.g., \cite{FPT-00},  \cite{M3-00}, \cite{N2-09}, and \cite{PR-04},), much less is known about the structure of (co)amoebas coming from varieties of higher codimension. A natural first step in this direction is to explore the case of linear spaces.

Being of a logarithmic nature, it is natural that  coamoebas are closely related to the exponents of the defining functions of
 $V$, and to the associated Newton polytopes. This connection is extensively explored in the thesis of the first author \cite{GKZ-94}, \cite{N1-10}, and \cite{PR-04}. 
Another important connection is to the currently very active field of tropical geometry, a piecewise linear incarnation of classical 
algebraic geometry where the varieties can be seen as non-Archimedean versions of amoebas (see \cite{MS}, \cite{M2-04}, \cite{M1-04} and others).

\noindent A fundamental theorem was shown by K. Purbhoo \cite{P-08}   for the general study of amoebas   that do not come from 
hypersurfaces.  The theorem   states  that the amoeba of an algebraic variety $V$ is equal to the intersection of all hypersurface amoebas corresponding to functions in the defining ideal $\mathcal{I}(V)$ of the variety $V$. 
We give a simple proof of this theorem with an extension to coamoebas.

\vspace{0.2cm}

\begin{theorem}\label{Main A} Let $V\subset (\mathbb{C}^*)^n$ be an algebraic variety with defining ideal $\mathcal{I}(V)$. Then the amoeba (respectively
 coamoeba) of $V$ is given as follows:
$$
\mathscr{A}(V) = \bigcap_{f\in \mathcal{I}(V)} \mathscr{A}(V_f) \qquad\qquad  and \qquad\,\,\,\,  co\mathscr{A}(V) = 
\bigcap_{f\in \mathcal{I}(V)} co\mathscr{A}(V_f) .
$$
\end{theorem} 

 In  \cite{PR-04}, Rullg\aa rd  and the second author  showed that the area of  complex plane curve amoebas is finite and the bound is given in terms of the Newton polygon of the defining polynomial. They,  also compute the area of the amoeba of a plane line. It was shown by  Mikhalkin and Rullg\aa rd that this bound is always sharp \cite{MR-00}. In \cite{MN1-11},
 Madani and the first author  generalized this result and showed  that the volume of the amoeba of  a  $k$-dimensional algebraic 
variety of $(\mathbb{C}^*)^n$  with $n\geq 2k$ is  finite. Moreover, they proved in \cite{MN2-11} that the finiteness of the volume of
 the amoeba of  a generic analytic variety is equivalent to  the variety being  algebraic.  Theorem \ref{Main A} and Proposition \ref{theorem dim amoeba} was shown separately and in the same time by Petter Johansson in \cite{J-10}.

\vspace{0.2cm} 
 
Let $V$ be a variety in the projective space $\mathbb{CP}^n$. We choose homogeneous coordinates $[Z_0:\cdots :Z_n]$ so that $V$ is transverse to coordinate hyperplanes $Z_j=0$ and all their intersections. The complement of the arrangement of  coordinate hyperplanes in  $\mathbb{CP}^n$ is $(\mathbb{C}^*)^n$. Then the variety  $\mathscr{V} = V\cap (\mathbb{C}^*)^n$ is called  a {\em very affine variety}, and in the case where $P(k)$ is a $k$-dimensional linear subspace of $\mathbb{CP}^n$ we say that $\mathscr{P}(k) = P(k)\cap (\mathbb{C}^*)^n$ is a {\em very affine linear space}, and by abuse of language we will call it just affine linear space. Moreover, $\mathscr{P}(k)$ can be presented as a
complete intersection of hyperplanes given by first degree equations $f_1(z)=\cdots =f_{n-k}(z)=0$, where $z=(z_1,\ldots ,z_n) = (Z_1/Z_0,\ldots ,Z_n/Z_0)$ stands for the affine coordinates in $(\mathbb{C}^*)^n$.

\vspace{0.2cm}
 
\begin{theorem}\label{main theorem6}
Let $\mathscr{P}(k)$ be a generic affine linear subspace of $(\mathbb{C}^*)^{2k}$. Then we have the following:
\begin{itemize}
\item[(i)]\, The volume of the coamoeba $co\mathscr{A}(\mathscr{P}(k))$ is equal to $\pi^{2k}$;
\item[(ii)]\,  Moreover, if $\mathscr{P}(k)$ is real, then the volume of its amoeba  $\mathscr{A}(\mathscr{P}(k))$ is equal to 
$\frac{\pi^{2k}}{2^k}$.
\end{itemize}
\end{theorem}

\vspace{0.2cm} 

The present paper is organized as follows. We give definitions, background, and some known results in connection with this paper in Section \ref{sec:2}. We prove Theorem \ref{Main A} in Section \ref{sec:3},
and detailed description of  amoebas and coamoebas  of lines in $n$-dimensional complex  algebraic torus  in Section \ref{sec:4} for any $n\geq 2$.   We prove  Theorem \ref{main theorem6} in Section \ref{sec:5}.

\vspace{0.2cm}

\noindent {\bf Remark.} My first meeting and mathematical discussion with Michael was during the summer school in Paris in 2006 where he gave a series of lectures on amoebas. We talked a lot on the geometric and topological properties of 
these objects in particular the solidness  of some of them. Moreover,  at Stokholm University, when I visited him in the same year, we discuss their similarity to other objects called coamoebas. At that time we do not know exactly what kind of similarities because the ambient  spaces of  these two objects are  different one is  compact and the other is no compact.
Amoebas are closed subsets in the Euclidean  space but coamoebas are not closed and not open subsets of the real torus. However,  both of them have a similar (dual in some sense) combinatorial properties, and strongly related to the combinatorial type of the Newton polytopes of the defining polynomial in  the  hypersurface case. At that time we do not know a lot of things in higher codimension.  This work was started on June  2011, but after the tragic death of Mikael Passare on 15 September 2011, the completion and writing of this paper was done by the first author.


\section{Preliminaries}\label{sec:2}

In this section, we review some known results related to this paper, and  give some notations and definitions. Let $V$ be an algebraic variety in $(\mathbb{C}^*)^n$. The {\em amoeba}  $\A$ of $V$ is by definition the image of $V$ under the logarithmic map  defined as follows (see M. Gelfand, M.M. Kapranov
 and A.V. Zelevinsky \cite{GKZ-94}):
\[
\begin{array}{ccccl}
\Log&:&(\mathbb{C}^*)^n&\longrightarrow&\mathbb{R}^n\\
&&(z_1,\ldots ,z_n)&\longmapsto&(\log |z_1|,\ldots ,\log|
z_n|).
\end{array}
\]

The argument map is the map  defined as follows:
\[
\begin{array}{ccccl}
\Arg&:&(\mathbb{C}^*)^n&\longrightarrow&(S^1)^n\\
&&(z_1,\ldots ,z_n)&\longmapsto&(\frac{z_1}{|z_1|},\ldots ,\frac{z_1}{|z_1|}).
\end{array}
\]
 The {\em coamoeba} of $V$,  denoted by $co\mathscr{A}$, is its image under the argument map (defined  by the second author in 2004). 

\vspace{0.1cm}

Purbhoo shows that the amoeba of an algebraic variety $V$ is equal to the intersection of all hypersurface amoebas corresponding to functions in the defining ideal $\mathcal{I}(V)$ of the variety $V$ (see \cite{P-08}, Corollary 5.2).
Passare and Rullg\aa rd  prove the following (see \cite{PR-04}):

\begin{theorem}[Passare-Rullg\aa rd,(2000)]\label{Passare-Rullgard}
 Let $f$ be a Laurent polynomial in two variables. Then the area of the amoeba of an algebraic plane  curve  with defining polynomial $f$ is not greater than $\pi^2$  times the area of the Newton polytope of $f$.
\end{theorem}

\vspace{0.05cm}

In \cite{MR-00}, Mikhalkin and Rullg\aa rd showed that up to multiplication by a constant in $(\mathbb{C}^*)^2$,  the algebraic plane curves with Newton polygon $\Delta$ with  maximal amoeba area are defined over $\R$. Furthermore,  their real loci are isotopic to the so-called Harnack curves (possibly singular with ordinary real isolated double points). Moreover,  Rullg\aa rd  and the second author compute the area of the amoeba of a line in the plane.

\vspace{0.1cm}

Madani and the first author showed that if the dimension $n$ of the ambient space  is at least the double of the dimension of $V$ (i.e., $n\geq 2\dim_{\mathbb{C}} (V) = 2k$), then the map  $\Log\circ\Arg^{-1}$ conserves the $2k$-volume, i.e., the absolute value of the determinant of its Jacobian, when it exists, is equal to one (see  \cite{MN2-11},  Proposition 3.1). Moreover, the same proposition shows that the set of critical points of the logarithmic and the argument maps restricted to $V$ coincide.
Hence, if the argument map restricted to the set of  regular points is injective, 
  and the cardinality  $d$ of  the inverse image under the logarithmic map of 
a regular value  in the amoeba is constant, then the volume of  the amoeba  will be the volume of the coamoeba divided  by $d$. So, first  we show that if $V$ is a generic $k$-dimensional linear space in $(\mathbb{C}^*)^{2k}$, then the argument map restricted to the set of  regular points is injective, and we compute the volume of its coamoeba.  Moreover, if the linear space is real,  we show  that  the cardinality of the inverse image under the logarithmic map of a regular value in the amoeba is constant and equal to $2^k$.   Finally, we compute the amoeba volume using the conservation of the volume by the map $\Log\circ\Arg^{-1}$.


%
%

\vspace{0.2cm}

\hspace{0.1cm}  In the following paragraph, we will recall the definition of the logarithmic Gauss map for hypersurface, and its generalization.  We will present some known relations between this map and (co)amoebas. Let  $V\subset (\mathbb{C}^*)^n$ be an algebraic  hypersurface with defining polynomial $f$, and denote by $V_{reg}$ the subset of its smooth points . The
  {\em logarithmic Gauss map} of the hypersurface $V$ is the holomorphic map defined by (see Kapranov \cite{K-91}):
\begin{eqnarray*}
\gamma :&V_{reg}&\longrightarrow\,\,\,\mathbb{CP}^{n-1}\\
&z&\longmapsto\,\,\, \gamma (z) =  [v (z)],
\end{eqnarray*}

\noindent where  $[v(z)]= [z_1\frac{\partial f}{\partial z_1}(z):\ldots :  z_n\frac{\partial f}{\partial z_n}(z)]$ denotes the class of the vector $v(z)$ in $\mathbb{CP}^{n-1}$.


\hspace{0.1cm} Madani and the first author generalize this map to any codimension, and extract  some  relations between the set of its critical points and (co)amoebas,  and  they generalized an earlier result of Mikhalkin \cite{M3-00} on critical points of the logarithmic map (see \cite{MN3-12}). More precisely, let 
 $V\subset (\mathbb{C}^*)^n$ be an algebraic variety of dimension $k$ with defining ideal $\mathcal{I}(V)$ 
generated by $\{ f_1,\ldots ,f_l\}$. A holomorphic map $\gamma_G$ from the set of smooth points of $V$ to the complex Grassmannian $\mathbb{G}_{n-k,\, n}$  was defined  as follows:
If we  denote by $V_{reg}$ the subset of smooth points of $V$ as before, and $M(l\times n)$ denotes the set of $l\times n$ matrices. Let 
$g_G$ be the following map:
\[
\begin{array}{ccccl}
g_G&:&V_{reg}&\longrightarrow&M(l\times n)\\
&&z=(z_1,\ldots ,z_n)&\longmapsto&
\left( \begin{array}{ccc} z_1\frac{\partial f_1}{\partial z_1}(z) & \ldots & z_n\frac{\partial f_1}{\partial z_n}(z) \\ 
\vdots & \vdots& \vdots \\ 
 z_1\frac{\partial f_l}{\partial z_1}(z)& \ldots &  z_n\frac{\partial f_l}{\partial z_n}(z)
 \end{array} \right) .
\end{array}
\]
 Since $z$ is a smooth point of $V$, then  the complex vector space $L_z$ generated by the rows of  the matrix $g_G(z)$ is of dimension $n-k$,
 and orthogonal to the tangent space to $V$ at $z$.   Indeed, the problem is local and $V_{reg}$ is locally a complete intersection. Moreover, the tangent space to $V$ at a regular point is contained in the tangent space of all the  hypersurfaces  defined by the polynomials $f_i$, and each row  vector of index $i$ is orthogonal to the  hypersurface defined by the polynomial $f_i$ which  contains $V$.
 This means that the image of $V_{reg}$ by $g_G$ is contained in the subvariety of 
$M(l\times n)$ consisting of $l\times n$ matrices of rank $n-k$, which we map to the complex Grassmannian $\mathbb{G}_{n-k,\, n}$.  Composing this identification with $g_G$
  we obtain the desired  map:
  $$
  \gamma_G : V_{reg}\rightarrow \mathbb{G}_{n-k,\, n}
  $$ 
   called the {\em generalized logarithmic Gauss map}.


If $V\subset (\mathbb{C}^*)^n$ is a hypersurface,  Mikhalkin showed that the set of critical points of $\Log_{|V}$ coincides with 
$\gamma^{-1}_G (\mathbb{RP}^{n-1})$ (see Lemma 3 in \cite{M3-00}, and Lemma 4.3 in \cite{M2-04}). This result was generalized by Madani and the first author for higher codimension in \cite{MN3-12}. 
%
%
%

Throughout all this paper, the genericity  of an algebraic variety $V\subset (\mathbb{C}^*)^n$ is defined as follows:


\begin{definition}
An irreducible   algebraic variety  $V\subset (\mathbb{C}^*)^n$  of dimension $k$ is generic if 
it satisfies the following: 
\begin{itemize}
\item[(1)]\, The variety $V$ contains an open dense subset $U$ such that the Jacobian   of the restriction to  $U$   of the logarithmic map   $\Jac (\Log_{| U})$ has maximal rank i.e., $\min \{ 2k, n\}$; 
\item[(2)]\, The variety $V$ lies in no affine subgroup,   otherwise we may replace $(\mathbb{C}^*)^n$ by the smallest affine subgroup containing $V$. 
\end{itemize}
\end{definition}


\vspace{0.2cm}

We denote by  $\mathscr{L}og_{|V}$ the complex logarithmic map, and $\Ree$ the real part of a complex vector. In this case, we have  
$\Log_{|V} = \Ree \circ \mathscr{L}og_{|V}$. This means that the amoeba of $V$ is the real part of $\mathscr{L}og_{|V}(V)$ (by taking the imaginary part we obtain the same  conclusion for the coamoeba),
\begin{equation}
\xymatrix{
V\subset (\mathbb{C}^*)^n\ar[rr]^{\mathscr{L}og_{|V}}\ar[dr]_{\Log_{|V}}&&\mathbb{C}^n  \supset \mathscr{L}og_{}(V)\ar[ld]^{\Ree}     
\cr
&\mathscr{A}(V)\subset \mathbb{R}^n.
}\nonumber
\end{equation}
We can check that for any $r\in \mathbb{R}^n$, the set $T_r := \Log^{-1}(r)$ is a $n$-dimensional real torus, and $r\in \mathscr{A}(V)$ if and only if $T_r\cap V\ne\emptyset$.


\section{(Co)amoebas of complex algebraic varieties}\label{sec:3}


In this section, we describe the amoeba (respectively coamoeba) of a complex variety  $V$ with defining ideal $\mathcal{I}(V)$ as the intersection of
 the amoebas (respectively coamoebas) of the  complex hypersurfaces with defining polynomials in $\mathcal{I}(V)$. 


The first part of Theorem \ref{Main A} concerning  amoebas was shown by    Purbhoo  in 2008 (see  Corollary 5.2 in \cite{P-08}). We present  a very simple proof of this fact, and extend it to coamoebas.

\vspace{0.2cm}

Our first observation, is the following proposition about the dimension of (co)amoebas:

\begin{proposition}\label{theorem dim amoeba}
 Let   $V\subset (\mathbb{C}^*)^{n}$ be an irreducible algebraic variety of dimension  $k$. Then, the dimension of the (co)amoeba $\mathscr{A}(V)$ of  $V$ satisfies the following:
$$
 \dim ((co)\mathscr{A}(V)) \leq  \min \{  2k , n   \} .
$$
In particular, if $V$ is generic, then the dimension of its amoeba  is  equal to $\min \{  2k , n   \}$.
\end{proposition}

\begin{proof}
The rank of the Jacobian of the logarithmic (respectively argument)  map restricted to  $V$
at a  regular point is equal to $\min \{  2k , n   \}$. So, the dimension of the (co)amoeba  cannot  exceed $\min \{  2k ,  n  \}$. Moreover,  if the dimension of the amoeba (respectively  coamoeba) of a $k$-dimensional  irreducible variety $V$ in $(\mathbb{C}^*)^n$ is strictly less than $\min \{  2k , n\}$, then the  map $\Ree$ is not an immersion (respectively submersion) if $n\geq 2k$ (respectively $n< 2k$). Hence,
the set of critical points of the  logarithmic   (respectively argument) map is equal to all the variety (see \cite{MN3-12} for more details about critical values of the logarithmic  Gauss map in higher codimension case). 
\end{proof}

 Let $V_f\subset (\mathbb{C}^*)^n$ be a hypersurface with defining polynomial $f$. Then, by definition, the amoeba
 of $V_f$ is the image by the logarithmic map of the subset $\mathscr{S}_f$ of $(\mathbb{R}_+^*)^n$ defined as follows:
$$
\mathscr{S}_f := \{  (x_1,\ldots ,x_n)\in (\mathbb{R}_+^*)^n |\, \exists \, z\in (\mathbb{C}^*)^n \,\textrm{such that} \,   x_i=|z_i|,\, \textrm{and}\, \, f(z)=0\} .
$$
Since  $\mathscr{L}og : (\mathbb{R}_+^*)^n \rightarrow \mathbb{R}^n$ is a diffeomorphism, we have the following:
$$
 \bigcap_{f\in \mathcal{I}(V)} \Log (\mathscr{S}_f)   =  \Log \left( \bigcap_{f\in \mathcal{I}(V)} \mathscr{S}_f \right),
$$
where $\Log (\mathscr{S}_f)$ is used with abuse of notation.
\begin{lemma}\label{lemma A} We have the following equality:
$$
\bigcap_{f\in \mathcal{I}(V)} \mathscr{S}_f 
= 
\{  (x_1,\ldots ,x_n)\in (\mathbb{R}_+^*)^n |\, \, x_i = |z_i|,\,\, {\rm and} \, \, (z_1,\ldots ,z_n)\in V \}.
$$
\end{lemma}

\begin{proof} Let $r$ be in 
 $$
 (\mathbb{R}_+^*)^n\setminus\{  (x_1,\ldots ,x_n)\in (\mathbb{R}_+^*)^n | \, x_i=|z_i|\, \, {\rm and} \, \,  (z_1,\ldots ,z_n)\in V \} ,
 $$
  and $T_r$ be the real torus $\Log^{-1}(r)$. So, $T_r\cap V$ is empty. Let $f\in \mathcal{I}(V)$ with $f(z)= 
\sum c_{\alpha}z^{\alpha}$ and $g$ be the Laurent polynomial defined by $g(z) = \sum \overline{c}_{\alpha}w^{\alpha}$ with
 $w=(\frac{r_1^2}{z_1},\ldots , \frac{r_n^2}{z_n})$ where the $r_j$'s are the coordinates of $r$, and $\overline{c}_{\alpha}$
 denotes the conjugate of the coefficient $c_{\alpha}$. The value of the 
  Laurent polynomial $h(z) = f(z)g(z)$  is equal to the value of $|f(z)|^2$  for  every  $z\in T_r$.
 By construction, the hypersurface $V_h$  with defining polynomial $h$ contains $V$ (because 
$h\in \mathcal{I}(V)$).
Let $\langle f_1,\ldots ,f_s\rangle$   be a set of generators of the ideal $\mathcal{I}(V)$, and  for any $j$ let $g_j$ be  the Laurent polynomial
 defined as before.  We can check  
 the hypersurface defined by the polynomial $ G=\sum f_jg_j$ contains $V$ and does not intersect the torus  $T_r$. This proves that
 $r\in  (\mathbb{R}_+^*)^n \setminus\bigcap_{f\in \mathcal{I}(V)} \mathscr{S}_f $. Hence, we have the inclusion:
$$
\bigcap_{f\in \mathcal{I}(V)} \mathscr{S}_f \subset \{  (x_1,\dots ,x_n)\in (\mathbb{R}_+^*)^n\, | \, x_i=|z_i|,\,\, {\rm and} \,\, 
( z_1,\ldots ,z_n)\in V \} .
$$
Now let $(x_1,\dots ,x_n)\in (\mathbb{R}_+^*)^n$ such that $x_i=|z_i|$ and $( z_1,\ldots ,z_n)\in V$, then  for all $f\in \mathcal{I}(V)$ we have   $f( z_1,\ldots ,z_n)=0$. This means that $(x_1,\dots ,x_n)\in \bigcap_{f\in \mathcal{I}(V)} \mathscr{S}_f$.
\end{proof}


 \noindent {\it Proof of Theorem \ref{Main A}.}  \,  The first equality of Theorem \ref{Main A} is a consequence of  Lemma \ref{lemma A}. In fact, by applying the logarithmic map to both sides of the equality of Lemma \ref{lemma A}
we obtain: $ \Log \left( \bigcap_{f\in \mathcal{I}(V)} \mathscr{S}_f \right)  = \mathscr{A}(V)$, 
and then 
$$
 \mathscr{A}(V)  = \bigcap_{f\in \mathcal{I}(V)} \mathscr{A}(V_f).
$$

Let us prove the  second equality of Theorem \ref{Main A}.  Let $w\in \bigcap_{f\in 
\mathcal{I}(V)} co\mathscr{A}(V_f) $, then there exists  a fundamental domain $\mathscr{D} = ([a;  a+2\pi [)^n$ in the universal 
covering of the real torus $(S^1)^n$ and  a unique $\widetilde{w}\in \mathscr{D}$ such that $w =\exp (i\widetilde{w})$. In this domain,
 the exponential map is a diffeomorphism between $\mathscr{D}$ and $(S^1)^n\setminus (S^1)^{n-1} \wedge\ldots \wedge (S^1)^{n-1}$ where $(S^1)^{n-1} 
\wedge\ldots \wedge (S^1)^{n-1}$ denotes  the bouquet of $n$  tori of dimension $n-1$. Let us define the subset $co\mathscr{S}_f$ of $\mathscr{D}$ as follow:
$$
co\mathscr{S}_f := \{ \theta\in \mathscr{D} | \,\, {\rm there \,\, exists}\,\,   z\in V_f\,\, {\rm and}\,\, \exp(i\theta) = \Arg (z)\} .
$$
So, we have:
$$
\bigcap_{f\in \mathcal{I}(V)}  \exp \left(  i co\mathscr{S}_f\right) = \exp \left(   i \bigcap_{f\in \mathcal{I}(V)} 
co\mathscr{S}_f\right)
$$
because the exponential map is a diffeomorphism from  $\mathscr{D}$ into its image. Moreover, $\widetilde{w}$ is contained in the intersection
 $\bigcap_{f\in \mathcal{I}(V)} co\mathscr{S}_f$. But the last intersection, using the same argument  as in Lemma \ref{lemma A}, can be described as
 follows:
\begin{eqnarray}
\bigcap_{f\in \mathcal{I}(V)} co\mathscr{S}_f &=&   \bigcap_{f \in \mathcal{I}(V)}  \{   \theta \in \mathscr{D}|\, \, {\rm there \,\, exists}\,\,   z\in V_f\,\, {\rm and}\,\, \exp(i\theta) = \Arg (z)  \}  \nonumber\\
&=&  \{   \theta \in \mathscr{D}|\, \, {\rm there \,\, exists}\,\,   z\in V\,\, {\rm and}\,\, \exp(i\theta) = \Arg (z)  \}.\nonumber
\end{eqnarray}
Indeed, to prove the last equality,  let $e^{i\theta} \notin co\mathscr{A}(V)$, and for each generator $f_j(z)=\sum c_{\alpha}z^{\alpha}$ of $\mathcal{I}(V)$ we 
define the polynomial $g_j$ as follows:
$$
g_j(z)=\sum \overline{c}_{\alpha}(e^{-2i\theta})^{\alpha}z^{\alpha}.
$$
If $z\in \Arg^{-1}(e^{i\theta})$,  then we have $f_jg_j(z)= |f_j(z)|^2$.  The polynomial 
$G=\sum_jf_jg_j$ is in $\mathcal{I}(V)$,  but $e^{i\theta} \notin co\mathscr{A}(V_G)$ because $|f_j(z)|^2>0$ and hence $G(z)=\sum_jf_jg_j(z)>0$ for every $j$ and every $z\in \Arg^{-1}(e^{i\theta})$.
Namely,  we have the following inclusion:
$$
\bigcap_{f\in \mathcal{I}(V)} co\mathscr{S}_f \subset \{   \theta \in \mathscr{D}|\, \, {\rm there \,\, exists}\,\,   z\in V\,\, {\rm and}\,\, 
\exp(i\theta) = \Arg (z)  \} .
$$
In other words, $\bigcap_{f\in \mathcal{I}(V)} co\mathscr{A}_f \subset co\mathscr{A}(V)$.\hspace{5cm} $\Box$


\section{(Co)Amoebas  of  linear spaces}\label{sec:4}

Throughout this section, $\mathscr{P}:= P(k) \cap (\mathbb{C}^*)^{k+m}$  where  $P(k)$ is  the $k$-dimensional 
  affine linear subspace of $\mathbb{C}^{k+m}$ given by the   parametrization  $\rho$  as follows:
\begin{equation}
\begin{array}{ccccl}
\rho&:&\mathbb{C}^k&\longrightarrow&\mathbb{C}^{k+m}\\
&&(t_1,\ldots ,t_k)&\longmapsto&(t_1,\ldots ,t_k,f_1(t_1,\ldots ,t_k),\ldots ,f_m(t_1,\ldots ,t_k)),
\end{array}
\end{equation}
where  $f_j(t_1,\ldots ,t_k) = b_j+\sum_{i=1}^ka_{ji}t_i$,  and  $a_{ji}$,\,  $b_j$ are  complex numbers for
 $i=1,\ldots , k$, and $j=1,\ldots , m$.  By abuse of language, we call $\mathscr{P}$ an affine linear space instead of very affine linear space.  First of all,  if $\mathscr{P}$ is generic then  all the coefficients  
$b_j$ are different than zero. 
 Otherwise $\mathscr{P}$ will be contained in an affine  subgroup of $(\mathbb{C}^*)^{k+m}$.
Indeed,  if there exits $j$ such that $b_j=0$, then 
  there is an action of $\mathbb{C}^*$ on $\mathscr{P}$, and then $\mathscr{P}$
can be viewed as a product of  
$\mathbb{C}^*$ with an affine linear space of dimension $k-1$.  Namely, $\mathscr{P}$ lies in no affine subgroup, i.e., $\rho (\mathbb{C}^k)$ meets each of the $n$ coordinate hyperplanes of $\mathbb{C}^n$ in distinct hyperplanes, otherwise we may replace $(\mathbb{C}^*)^n$ by the smallest affine subgroup containing $\mathscr{P}$.

\begin{lemma}
If $\mathscr{P}$ is generic, then we can assume that $f_1(t_1,\ldots ,t_k) = 1+\sum_{i=1}^kt_i$.
\end{lemma}

\begin{proof} 
In fact,  if we  make  a translation by 
$\frac{1}{b_1}$ in the algebraic multiplicative torus $(\mathbb{C}^*)^{k+m}$, we get 
 $$
 \biggl(\frac{t_1}{b_1},\ldots ,   \frac{t_1}{b_1}, \frac{f_1(t_1,\ldots ,t_k)}{b_1} , \ldots ,\frac{f_m(t_1,\ldots ,t_k)}{b_1}\biggr).
 $$ 
 We translate  again by $a = (a_{11},a_{21},\ldots ,a_{1k},1,\dots ,1)$ to obtain:
$$
\biggl(\frac{a_{11}t_1}{b_1},\ldots \frac{a_{1k}t_k}{b_1}, 1+\sum_{i=1}^k\frac{a_{1i}t_i}{b_1}, \frac{f_2(t_1,\ldots ,t_k)}{b_1},\ldots , 
\frac{f_m(t_1,\ldots , t_k)}{b_1}\biggr).
$$
For any point $z$ in $(\mathbb{C}^*)^{k+m}$,  we  denote by $\tau_z$ the translation  by $z$ in the multiplicative group  $(\mathbb{C}^*)^{k+m}$, and 
denote by $\rho'$ the required parametrization, i.e., 
$$
\rho' (t_1,\ldots ,t_k) = \biggl(t_1,\ldots ,t_k,1+\sum_{i=1}^kt_i, f_2(t_1,\ldots ,t_k), \ldots ,f_m(t_1,\ldots ,t_k)\biggr).
$$
Hence,  we obtain $\tau_a\circ\tau{\frac{1}{b_1}}\circ\rho = \rho'\circ\tau_c$, where $c=(\frac{a_{11}}{b_1},\ldots ,\frac{a_{1k}}{b_1})$, and then,
for any $(t_1,\ldots ,t_k)$ in $(\mathbb{C}^*)^k$ we have:
$$
\Arg \biggl(\rho (t_1,\ldots ,t_k)\biggr)  - \Arg (b_1) + \Arg (a) = \Arg \biggl(\rho' (\tau_c(t_1,\ldots ,t_k))\biggr).
$$
We obtain the same relation if we replace the argument map by the logarithmic map.
This means that  the amoeba (respectively coamoeba) of a generic  complex  affine linear space  $\mathscr{P}$ given by the parametrization $(1)$ 
is the translation in the real space $\mathbb{R}^{k+m}$  (respectively the real torus $(S^1)^{k+m}$) by a vector  $v$ in $\mathbb{R}^{k+m}$ (respectively
 a point in the real  torus) of an affine linear space given by a 
parametrization such that $f_1(t_1,\ldots ,t_k) = 1+\sum_{i=1}^kt_i$.
Hence,
$co\mathscr{A}(\mathscr{P}) = \tau_v\circ co\mathscr{A}(\mathscr{P}_{\rho'})$ where $\mathscr{P}_{\rho'}$ is the affine linear
 space  given by the required  parametrization, and we have a similar equality for their  amoebas. 
 In the last formula, $v$  is the
 argument  of the vector $b_1^{-1}a$. 
 \end{proof}
 
\hspace{0.1cm} To be more precise, $\mathscr{P}$ can be seen as the image by $\rho$ of the complement in $\mathbb{C}^k$ of an arrangement of $n$ hyperplanes $\mathscr{H}:= \cup_{i=1}^k\{ t_i= 0\}\cup_{j=1}^m\{f_j = 0\}$.


\subsection{(Co)Amoebas of lines in $(\mathbb{C}^*)^{1+m}$}\label{sec:4}

 In this subsection  we give  a complete description of  (co)amoebas of  generic lines in  $(\mathbb{C}^*)^{1+m}$ (we mean a complex subvariety  of complex dimension one defined by an ideal generated by  polynomials of degree one). Moreover, we describe the (co)amoebas of 
 real  lines, i.e., lines those are invariant under the involution given by the conjugation of complex numbers.  In other word,  lines given by a parametrization with real coefficients.
  But first,  let $L$ be a generic  line in $(\mathbb{C}^*)^{1+m}$ parametrized as follows:
 \begin{equation}
\begin{array}{ccccl}
\rho&:&\mathbb{C}^*&\longrightarrow&(\mathbb{C}^*)^{1+m}\\
&&t&\longmapsto&(t, t+1, a_2t+b_2, \ldots , a_mt+b_m),
\end{array}
\end{equation}
where $a_j$ and $b_j$ are  non vanishing  complex numbers.  

\begin{lemma}\label{Line} 
There are two types of amoebas of lines in $(\mathbb{C}^*)^{1+m}$ for $m\geq 3$. There are amoebas with boundary and other without boundary (we mean topological boundary). 
The amoebas  of  generic  lines given by  the parametrization $(2)$ have boundary  if and only if  $\frac{a_i}{b_i}\in \mathbb{R}^*$ for all $j=2,\ldots , m$. 
\end{lemma}

\begin{proof}
Since the boundary of an amoeba is a subset of the set of critical values of the logarithmic map, then an amoeba has a boundary means  that the set of critical points of the logarithmic map restricted to the variety is nonempty (see \cite{MN3-12}, and \cite{M3-00} for more details about the critical points).
 The Jacobian of the logarithmic map restricted to the line $L$ is given by:
$$
 \Jac (\Log_{|L})(t) = \frac{\partial \Log}{\partial (t,\bar{t})}  = \frac{1}{2} \left( 
 \begin{array}{cc}
 1/t&1/\bar{t}\\
 1/(t+1)&1/(\bar{t}+1)\\
a_2/(a_2t+b_2)&\bar{a}_2/(\overline{a_2t+b_2})\\
 \vdots&\vdots\\
a_m/(a_mt+b_m)&\bar{a}_m/(\overline{a_mt+b_m})
 \end{array}\right).
$$
 Hence, a point $\rho (t)$ is critical for $\Log_{|L}$ if and only if all the $2\times 2$-minors of the Jacobian matrix have determinant equal to zero. Let us write down these relations.
 The determinant of the  $2\times 2$-minor given by the two first rows:
$$
\frac{1}{2}  \left( 
 \begin{array}{cc}
 1/t&1/\bar{t}\\
 1/(t+1)& 1/(\bar{t}+1)
 \end{array}\right)
 $$
 is equal to zero, means the following equality holds:  
 $$
 \frac {1}{t} \frac{1}{\bar{t}+1} = \frac{1}{\bar{t}} \frac{1}{t+1}.
 $$ 
 
This  implies that $t$ should be real. For all $i= 2,\ldots m$, the $2\times 2$-minor:
$$
\frac{1}{2} \left( 
 \begin{array}{cc}
 1/t&1/\bar{t}\\
a_i/(a_it+b_i)&\bar{a}_i/(\overline{a_it+b_i})
 \end{array}\right)
 $$
gives the following relation:
 $$
 \frac {1}{t} \frac{\bar{a}_i}{(\overline{a_it+b_i})} = \frac{1}{\bar{t}} \frac {a_i}{a_it+b_i}.
 $$
 
 But $t$ is real, so $\bar{a}_i(a_it+b_i) = a_i(\overline{a_it+b_i})$, and hence 
 $\frac{a_i}{b_i} = \overline{(\frac{a_i}{b_i})}$, i.e., $\frac{a_i}{b_i}\in \mathbb{R}^*$ (because $L$ is generic,  all the coefficients are different than zero).  So, if  $\frac{a_i}{b_i}\in \mathbb{R}^*$ for $i=2,\ldots , m$, then
the set of critical points of $\Log_{|L}$ is the image under $\rho$ of the real part of $\mathbb{C}^*$, where this image intersects $(m+2)$ quadrants of $\mathbb{R}^{1+m}$ because $L$ is generic.
Moreover, this shows that the set of critical values of $\Log_{|L}$ is the image under $\Log \circ \rho$ of the real part of  $\mathbb{C}^*$, and the number of its connected components is $(m+2)$.
So, a  generic complex line given by the  parametrization $(2)$ with  $\frac{a_i}{b_i}\in \mathbb{R}^*$ for $i=2,\ldots , m$ is real up to a translation by a complex number, and its amoeba is a surface with boundary, and the boundary has  $(m+2)$ connected components. Also, we can check in this case that   the cardinality of the inverse image of a regular  (respectively critical) value is two (respectively one). 
\end{proof}

This motivates the following definition (see \cite{MR-00} for real plane curves):

\begin{definition}
A  generic affine line given by the following parametrization:
\begin{equation}
\begin{array}{ccccl}
\rho&:&\mathbb{C}^*&\longrightarrow&(\mathbb{C}^*)^{1+m}\\
&&t&\longmapsto&(t, a_1t+b_1, a_2t+b_2, \ldots , a_mt+b_m),
\end{array}
\end{equation}
where $a_j$ and $b_j$ are in $\mathbb{C}^*$
 is called  real  up to a translation by a vector in $(\mathbb{C}^*)^{1+m}$ if and only if
 $[\frac{a_{1}}{b_1}: \ldots  :\frac{a_{m}}{b_m}]\in \mathbb{R}\mathbb{P}^{m-1}$.
\end{definition}
If a line $L$ in $(\mathbb{C}^*)^{1+m}$ with $m\geq 2$ is not real, then its amoeba is a surface without boundary homeomorphic to the Riemann sphere without $(m+2)$ points  (see proof of Lemma \ref{Line}), and the map $\Log_{|L}$ is a  one-to-one map.

\vspace{0.2cm}

\hspace{0.1cm} The following lemma  gives a description of the coamoeba of a generic line in $(\mathbb{C}^*)^{1+m}$ with $m\geq 1$
\begin{lemma}
Let $L\subset (\mathbb{C}^*)^{1+m}$ be a generic   line  given by the parametrization $(3)$. The restriction of the argument map to the set of its regular points in $L$ is injective, and the inverse image under the argument map of a critical value has real dimension one. 
\end{lemma}

\begin{proof}
To see injectivity, let $(e^{i\theta}, e^{i\psi_1}, \ldots , e^{i\psi_m})$ be a fixed  regular value in $co\mathscr{A}(L)$. In other word, we have  $t= |t|e^{i\theta}$, and $f_j(t) = (a_jt+b_j) = |a_jt+b_j|e^{i\psi_1}$ for $j=1, \ldots , m$, and consider  $a_jt,\,  b_j$, and $f_j(t)$ as a vectors in the complex plane. Hence, for each $j=1, \ldots , m$ we obtain a parallelogram with vertices the origin,   and the extremities of the three vectors $a_jt,\,  b_j$, and $f_j(t)$. If one of these vectors is fixed, and the arguments of the
two others are fixed (which is our case, because $b_j$ is given and the arguments of $a_jt$ and $f_j(t)$ are fixed by assumption), then there exists at most  one parallelogram with those vertices. This implies the injectivity. 

\vspace{0.1cm}

\hspace{0.1cm} The second part of the lemma comes from the fact that the set of critical points of the logarithmic map and the argument map coincide (see Proposition 3.1 in \cite{MN2-11}). Indeed, the set of critical points is equal to $(m+2)$ connected components of dimension one (each one corresponds to the intersection of the real part of $L$ with some quadrant of $(\mathbb{R}^*)^{m+1}$).
\end{proof}

\vspace{0.3cm}

The set of critical points  of the argument map restricted to $L$  given by the parametrization $(3)$  is the image by $\rho$ of the real part of $\mathbb{C}^*$ translated  by $(1, b_1,\ldots , b_m)$ in $(\mathbb{C}^*)^{1+m}$ as a multiplicative group. So, the set of critical values  consists of the translation by $(1, \frac{b_1}{|b_1|}, \ldots , \frac{b_m}{|b_m|})$  of  $(m+2)$ points in the real torus  $(S^1)^{1+m}$ from the $2^{m+1}$  real points corresponding to the arguments of the $2^{m+1}$ quadrants of $\Ree ((\mathbb{C}^*)^{1+m}) = (\mathbb{R}^*)^{1+m}$.  The closure of the coamoeba of $L$ contains an arrangement of $(m+1)$ geodesic circles. Each  circle corresponds to an end of the line (i.e., where $L$ meets a coordinate axis). The union of these  circles is the set of accumulation points of arguments of sequences in $L$ with unbounded logarithm, and is called the phase limit set of $L$ (see \cite{NS1-11}  for more details). It is the counterpart of the logarithmic limit set introduced by Bergman in 1971 (see \cite{B-71} and \cite{MS} for more details), which consists of $(m+2)$ points is our case. In Figure 1 (respectively Figure 2), we draw the amoeba and the coamoeba of a real  (respectively non real) line in $(\mathbb{C}^*)^3$. The coamoebas in Figure 1, and Figure 2 are made  with collaboration with F. Sottile.

\vspace{0.3cm}

\begin{figure}[h!]
\begin{center}
\includegraphics[angle=0,width=0.7\textwidth]{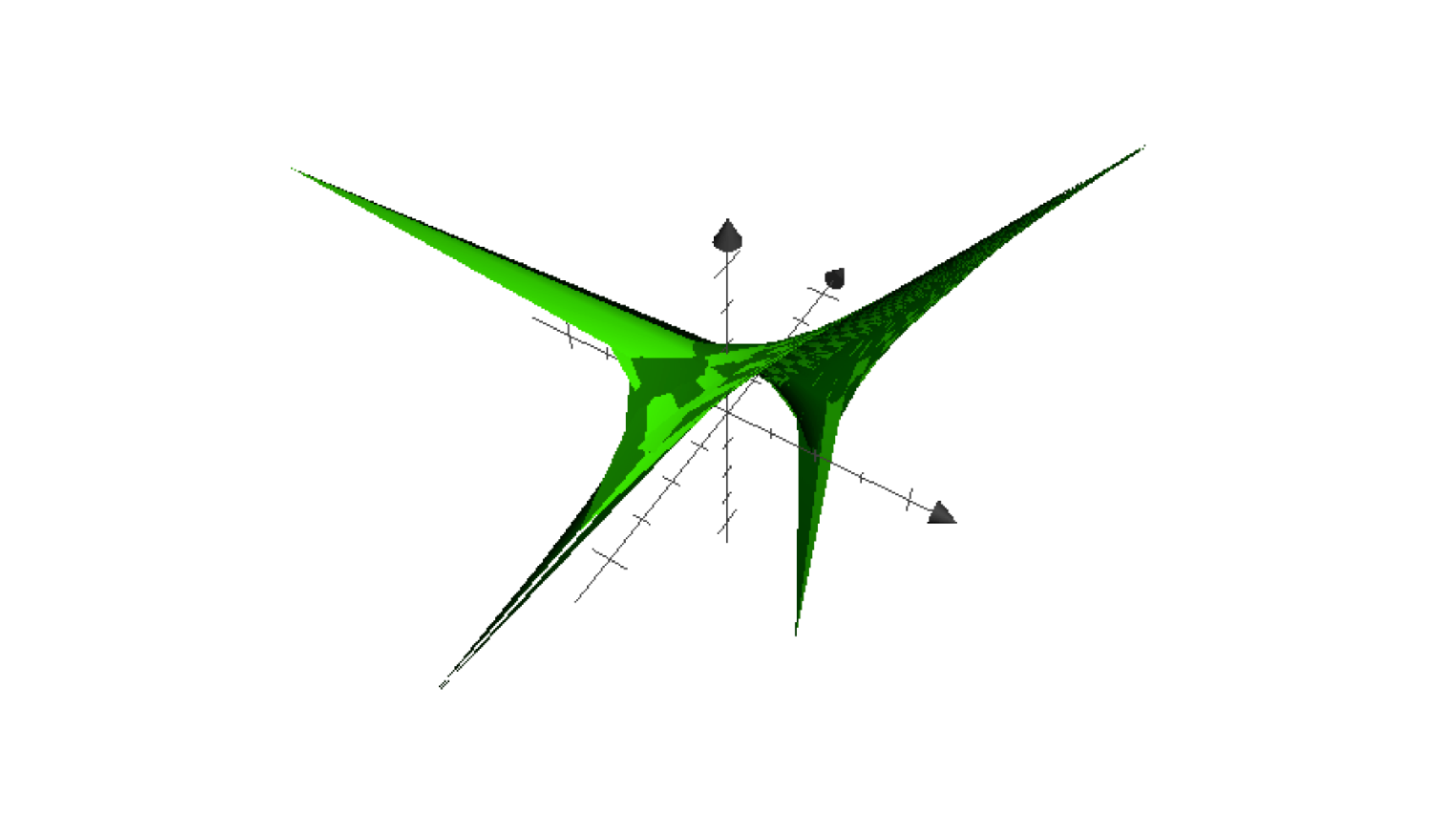}
\includegraphics[angle=0,width=0.27\textwidth]{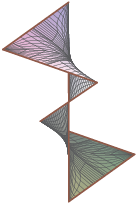}\qquad\qquad
\caption{The amoeba and the coamoeba of the real line in $(\mathbb{C}^*)^3$ given by the parametrization 
$\rho (z)=(z,z+\frac{1}{2},z-\frac{3}{2})$.  The amoeba is topologically   the closed disk without four points of its boundary.}
\label{c}
\end{center}
\end{figure}

\vspace{0.1cm}

\begin{figure}[h!]
\begin{center}
\includegraphics[angle=0,width=0.65\textwidth]{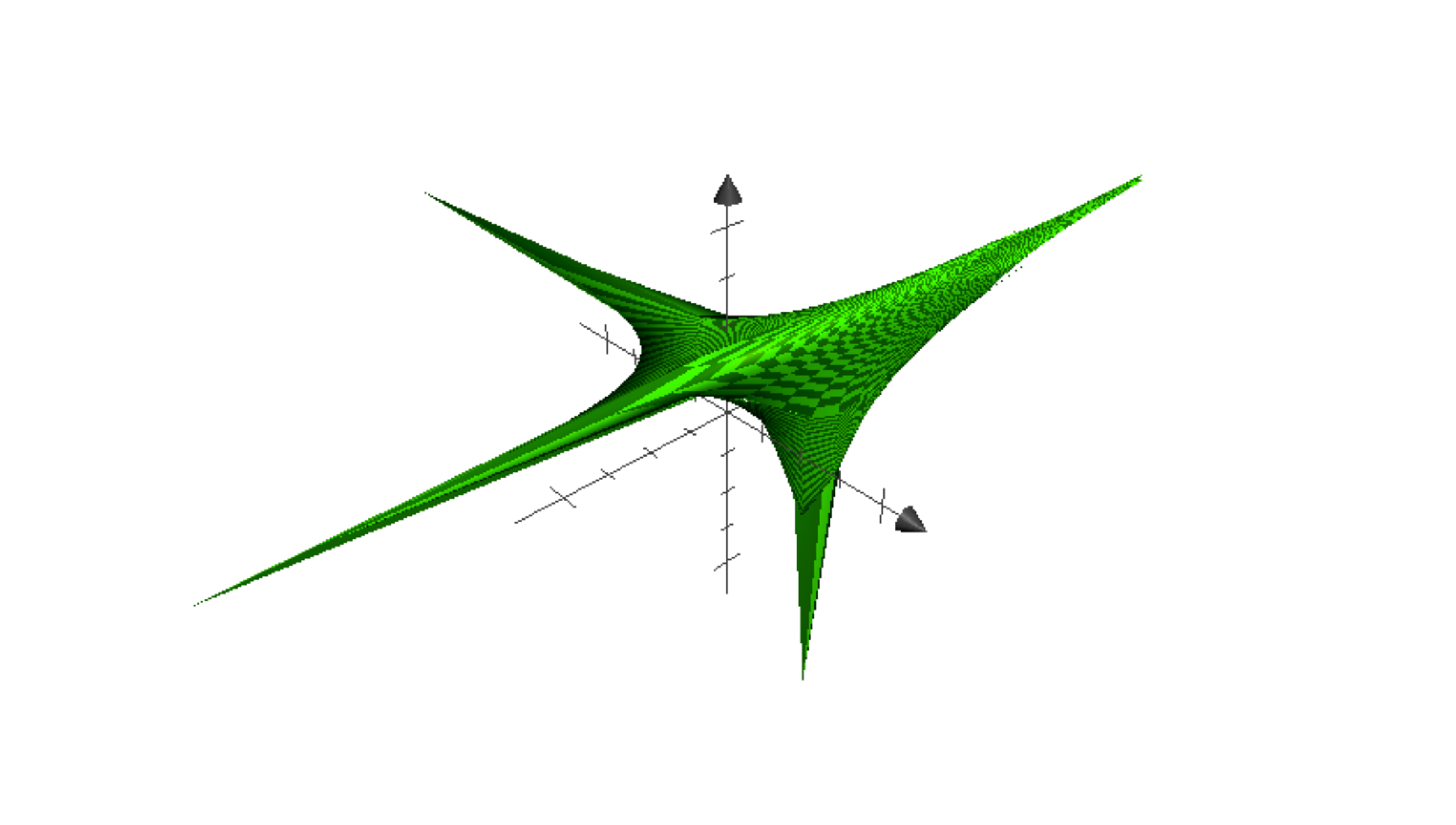}
\includegraphics[angle=0,width=0.3\textwidth]{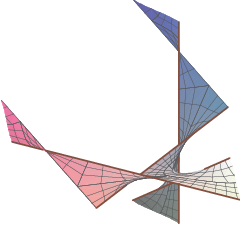}\qquad\qquad
\caption{The amoeba and the coamoeba of the complex line (i.e., not real) in $(\mathbb{C}^*)^3$ given by the parametrization 
$\rho (z)=(z,z+1,z-2i)$. The amoeba is topologically  the Riemann sphere without four points.}
\label{c}
\end{center}
\end{figure}


\vspace{0.1cm}

\section{Volume of (co)amoebas  of  $k$-dimensional very affine linear spaces in $(\mathbb{C}^*)^{2k}$}\label{sec:5}

\vspace{0.1cm}

It was shown by  Rullg\aa rd and the second  author   in \cite{PR-04} that the area of the amoeba of a complex  algebraic plane curve  is always finite,  and the bound is given in terms of the area of the Newton polygon of the defining polynomial. Mikhalkin and Rullg\aa rd proved that this  bound is always sharp for (possibly singular) Harnack curves (see \cite{MR-00}). It was shown by Madani  and the first author in \cite{MN1-11}  that the    volume of the amoeba of  a  $k$-dimensional algebraic variety in $(\mathbb{C}^*)^n$  with $n\geq 2k$ is finite. This generalizes the result of Rullg\aa rd and the second author about the finiteness of the  volume of the amoeba of plane curves. 
In this section, we compute the volume of the amoeba of a generic real $k$-dimensional  very
affine linear space in $(\mathbb{C}^*)^{2k}$. We will proceed as follows: (i)    We show  that the argument map restricted to the subset of   regular points in the very affine linear space is injective; (ii) We compute the volume of the coamoeba of any $k$-dimensional  very affine linear space in $(\mathbb{C}^*)^{2k}$; (iii) 
 We  compute the cardinality of the inverse image under the logarithmic map of any regular value  in the amoeba of a real affine space, and prove that this cardinality  is a constant and equal to $2^k$; 
(iv) We use that the map $\Log\circ\Arg^{-1}$ conserves the volume i.e., the determinant of its  Jacobian  has absolute value equal one (see Proposition 3.1 in \cite{MN2-11}), and  finally we compute the volume of the amoeba, which is  equal to  the coamoeba  volume divided by $2^k$ if the plane  is real. 
We will use the following lemma  proved in \cite{MN3-12}, which is a generalization of Mikhalkin's  Lemma 4.3 in \cite{M2-04} for hypersurface:

\begin{lemma}[Madani-Nisse]\label{Lem MN1}
Let $V\subset (\mathbb{C}^*)^{n}$ be a $k-$dimensional  algebraic variety,  and $z$ be a smooth point of $V$. Then 
$z$ is  a critical point  for the map  $\Log_{| V}$ if and only if the image of the 
tangent space $T_zV$ to $V$ at $z$ by the derivative of the complex logarithm $d\mathscr{L}og$ contains at least $s$ purely  imaginary linearly  independent vectors with $s = \max \{ 1, 2k-n+1\}$.
\end{lemma}

\vspace{0.2cm}

Also, we will use the following proposition proved in \cite{MN3-12}:

\begin{proposition}[Madani-Nisse]\label{Pro MN1}
Let $\mathscr{P}\subset (\mathbb{C}^*)^n$ be a generic $k$-dimensional  very affine linear space with $n\geq 2k$. Suppose that the complex
 dimension of $\mathscr{P}\cap \overline{\mathscr{P}}$ is equal to $l$, with $0\leq l\leq k$. Then, for any regular value $x$ in
 the amoeba $\mathscr{A}(\mathscr{P})$ of $\mathscr{P}$,  the cardinality of $\Log^{-1}(x)$ is at least $2^l$.
\end{proposition}

\vspace{0.2cm}

Let $\mathscr{P}\subset (\mathbb{C}^*)^{2k}$ be  a generic $k$-dimensional   very affine
linear space. Suppose $\mathscr{P}$ is given by the  parametrization $\rho$:
\begin{equation}
\begin{array}{ccccl}
\rho&:&(\mathbb{C}^*)^k&\longrightarrow&(\mathbb{C}^*)^{2k}\\
&&(t_1,\ldots ,t_k)&\longmapsto&(t_1,\ldots ,t_k,f_1(t_1,\ldots ,t_k),\ldots ,f_k(t_1,\ldots ,t_k)),
\end{array}
\end{equation}
 with  $f_j(t_1,\ldots ,t_k) = b_j+\sum_{i=1}^ka_{ji}t_i$,  where  $a_{ji}$, and $b_j$ are  complex numbers for $i=1,\ldots , k$  
and $j=1,\ldots , k$. Since the space $\mathscr{P}$ is generic, then there is no $b_j=0$.

\vspace{0.2cm}

\begin{definition}
A generic  $k$-dimensional   very affine linear space  $\mathscr{P}(k)\subset (\mathbb{C}^*)^{2k}$ 
given by the   parametrization $(4)$
is said to be real up to a translation by a complex vector in the multiplicative group $(\mathbb{C}^*)^{k+m}$ if and only if the $(m\times k)$-matrix given by
$$
\left( \begin{array}{ccc} \frac{a_{11}}{b_1} & \ldots & \frac{a_{1k}}{b_1} \\ 
\vdots & \vdots& \vdots \\ 
  \frac{a_{k1}}{b_k}  & \ldots &   \frac{a_{kk}}{b_k} 
 \end{array} \right) 
$$
has rank $k$ and all of its entries are real. 
\end{definition}

Let $\mathbb{Z}_2 := \{ \pm 1\}$  be the real subgroup of the multiplicative group  $\mathbb{C}^*$, and $\mathbb{Z}_2^{2k}$  be the finite real subgroup of $(\mathbb{C}^*)^{2k}$. For each 
$s\in\mathbb{Z}_2^{2k}$, let $\rho_s$ be the parametrization given by $\rho_s(t_1,\ldots ,t_k) = s.\rho (t_1,\ldots ,t_k)$ where $s.(z_1,\ldots ,z_{2k})$ $ = (s_1z_1,\ldots ,s_{2k}z_{2k})$ for any $(z_1,\ldots ,z_{2k})\in (\mathbb{C}^*)^{2k}$, and $s=(s_1,\ldots ,s_{2k}) \in \mathbb{Z}_2^{2k}$.
 Let  $\mathscr{P}_s$  be the  $k$-dimensional  very affine linear space in $(\mathbb{C}^*)^{2k}$
 parametrized by $\rho_s$. Let  us denote by $Reg (co\mathscr{A}(\mathscr{P}_s))$ the set of regular values of $co\mathscr{A}(\mathscr{P}_s)$. Remark that if $1$ denotes the identity element of the group $\mathbb{Z}_2^{2k}$, then $\mathscr{P} = \mathscr{P}_1$.

 \vspace{0.2cm}


\hspace{0.1cm } Let $u\in\mathbb{Z}_2^{2k}$ and  denote by  $Reg (co\mathscr{A}(\mathscr{P}_u))$ 
the set of regular values of the coamoeba $co\mathscr{A}(\mathscr{P}_u)$.

\begin{proposition} \label{main prop}
With the above notations,  the following statements hold:
\begin{itemize}
\item[(i)]\,  For all $s$, the argument map from the subset of regular points of $\mathscr{P}_s$ to the set of regular values of its coamoeba
 $co\mathscr{A}(\mathscr{P}_s)$ is injective;
\item[(ii)]\, Let $s$ and $r$  in  $\mathbb{Z}_2^{2k}$ with $s\ne r$, then the set 
$$
Reg (co\mathscr{A}(\mathscr{P}_s))\cap Reg (co\mathscr{A}(\mathscr{P}_r))
$$
is empty;
\item[(iii)]\, The union
$
\bigcup_{s\in \mathbb{Z}_2^{2k}}Reg (co\mathscr{A}(\mathscr{P}_s))
$
is an open dense subset of the real torus $(S^1)^{2k}$.
\end{itemize}
\end{proposition}

\vspace{0.3cm}
 
\hspace{0.1cm} First of all, we denote by $z:= (z_1,\ldots, z_{2k})$ the coordinates of $\mathbb{C}^{2k}$. So, if $z$ is a point in $\mathscr{P}$, then $z_i=t_i$   and $z_{k+i} = f_i(z_1,\ldots,z_k)$ for $1\leq i\leq k$.
Let $\Theta = (e^{i\theta_1},\ldots ,e^{i\theta_k}, e^{i\psi_1},\ldots , e^{i\psi_k})$ be a point in the set of regular values of 
$co\mathscr{A}(\mathscr{P})$.
 This means that  the linear system  $(E)$ of $2k$ equations and $2k$ variables 
$(x_1,\ldots , x_k, y_1,\ldots , y_k)$ in $(\mathbb{R}_+^*)^{2k}$:
 $$
\left\{ \begin{array}{ccc}
\Ree (b_j +\sum_{l=1}^ka_{jl}x_le^{i\theta_l})& = &\Ree\,  (y_je^{i\psi_j})\\
\Ima  (b_j +\sum_{l=1}^ka_{jl}x_le^{i\theta_l})& = &\Ima\, (y_je^{i\psi_j}) 
\end{array}\right.  \quad\quad\quad\quad \quad\quad  (E)
$$ 
with $j=1\ldots , k$,  has a solution in $(\mathbb{R}_+^*)^{2k}$.  Moreover, if $\mathbb{Z}_2^{2k}$ is viewed as a subgroup of the real torus $(S^1)^{2k}$, then  $s.\Theta \in \bigcup_{u\in \mathbb{Z}_2^{2k}}Reg (co\mathscr{A}(\mathscr{P}_u(k)))$  means that  the
 system $(E)$ has a solution in $(\mathbb{R}^*)^{2k}$.
 
 Since the matrix $A(z)$ defined by:
\begin{displaymath}
A(z) = \left( 
 \begin{array}{ccccccccc}
 a_{11}z_1& a_{12}z_2&\ldots& a_{1k}z_k&-z_{k+1}&0&0&\ldots&0\\
 a_{21}z_1& a_{22}z_2&\ldots& a_{2k}z_k&0&-z_{k+2}&0&\ldots&0\\
 \vdots& \vdots& \vdots& \vdots& \vdots& \vdots& \vdots& \vdots& \vdots\\
 a_{k1}z_1& a_{k2}z_2&\ldots& a_{kk}z_k&0&0&0&\ldots&-z_{2k}\\
 \end{array}\right)
 \end{displaymath}
is the image under the logarithmic Gauss map of the point $z$ in $\mathscr{P}$,  and the matrix $A(z)$ has  rank $k$ when $z$ is a regular point.

\vspace{0.3cm}

\noindent {\it Claim} I.  
If   $\overline{A}$  denotes the matrix conjugate to $A$, then for any regular point $z$ of $\mathscr{P}$ the matrix $\widehat{A}(z) = \left( \begin{array}{c} A(z)\\ \overline{A}(z)\end{array} \right)$  is of  rank $2k$.

\vspace{0.1cm}

\begin{proof}
  In fact, the rows of the matrix  $A(z)$ form a basis of the orthogonal space to $\mathscr{L}og (\mathscr{P})$  at  the point $\mathscr{L}og (z)$. So, if the rank of $\widehat{A}(z)$  is less than  $2k$, then the orthogonal space to $\mathscr{L}og (\mathscr{P})$  at $\mathscr{L}og (z)$ contains at least one real vector $v$ different than zero. This is equivalent to saying that the tangent  space to $\mathscr{L}og (\mathscr{P})$  at $\mathscr{L}og (z)$ contains at least one purely imaginary vector.   Indeed, since  $v$ is a vector different than zero orthogonal to  both $T_{\mathscr{L}og (z)}(\mathscr{L}og (\mathscr{P}))$ and $\Ima (\mathbb{C}^{2k})$, then $T_{\mathscr{L}og (z)}(\mathscr{L}og (\mathscr{P}))\cap\Ima (\mathbb{C}^{2k})$ must be of dimension at least one.
By Lemma \ref{Lem MN1}, this implies that $z$ is a critical point for the logarithmic map, which is in contradiction with our assumption on $z$.
\end{proof}
 The matrix defining the system $(E)$ is
$\widetilde{B}(\Theta ) = \left( \begin{array}{c} \Ree B(\Theta )\\ \Ima B(\Theta )\end{array} \right)$ where $B(\Theta )$ is  
\begin{displaymath}
 \left( 
 \begin{array}{ccccccccc}
a_{11}e^{i\theta_1}&a_{12}e^{i\theta_2}&\ldots&a_{1k}e^{i\theta_k}&-e^{i\psi_1}&0&0&\ldots&0\\
 a_{21}e^{i\theta_1}&a_{22}e^{i\theta_2}&\ldots&a_{2k}e^{i\theta_k}&0&-e^{i\psi_2}&0&\ldots&0\\
\vdots&\vdots&\vdots&\vdots&\vdots&\vdots&\vdots&\vdots&\vdots\\
a_{k1}e^{i\theta_1}&a_{k2}e^{i\theta_2}&\ldots&a_{kk}e^{i\theta_k}&0&0&0&\ldots&-e^{i\psi_k}
 \end{array}\right) .
 \end{displaymath}
We can check that the  rank of $\widetilde{B}(\Theta )$ is the same as the rank of the matrix 
$\widetilde{A}(z)=\left( \begin{array}{c} \Ree A(z)\\ \Ima A(z)\end{array} \right)$ with
 $z = (x_1e^{i\theta_1},\ldots ,x_ke^{i\theta_k}, y_1e^{\psi_1},\ldots , y_ke^{i\psi_k})$, because the variables $x_i$ and $y_j$ are non zero for all $i, j = 1, \ldots , k$.

\vspace{0.3cm}

\noindent {\it Claim} II.  
The rank of the matrix $\widetilde{A}(z)$ is  equal to $2k$.

\vspace{0.1cm}

\begin{proof}

 Suppose we have a non trivial linear combination of the rows of the matrix $\widetilde{A}(z)$ that is   equal to zero. Hence, there exist
 a real numbers 
$\lambda_l$, and $\mu_l$ not all equal to zero, with $l=1,\ldots ,k$ such that:
$$
\sum_{l,j=1}^k \frac{\lambda_l}{2}\biggl ( (z_ja_{lj}+\bar{z}_j\bar{a}_{lj}) -(z_{k+l} +\bar{z}_{k+l})  \biggr) + 
\frac{\mu_l}{2i}\bigg( (z_ja_{lj}-\bar{z}_j\bar{a}_{lj}) -(z_{k+l} -\bar{z}_{k+l})\biggr) = 0.
$$

We get:
$$
\sum_{l=1}^k\biggl(\frac{\lambda_l-i\mu_l}{2}\biggr)\biggl(\sum_{j=1}^kz_ja_{lj} -z_{k+l}\biggr) + 
\sum_{l=1}^k\biggl(\frac{\lambda_l+i\mu_l}{2}\biggr)\biggl(\sum_{j=1}^k\bar{z}_j\bar{a}_{lj} -\bar{z}_{k+l}\biggr) =0.
$$
Since the matrix $\widehat{A}(z)$ is of rank $2k$ by Claim I,  this implies that $\lambda_l-i\mu_l=0$, and $\lambda_l+i\mu_l =0$ for all
 $l=1,\ldots , k$. This means that all the $\lambda_l$'s and the $\mu_l$'s vanish. This contradict the fact that some of  the real numbers $\lambda_l$'s and $\mu_l$'s are different than zero by hypothesis. Hence, the real rank of the matrix 
$\widetilde{A}(z)$ is  equal to $2k$.
\end{proof}


\noindent {\it Proof of Proposition \ref{main prop}}.   Since the  $k$-dimensional linear space $\mathscr{P}$ is generic, the coefficients $b_j$ are different than zero, and  the system $(E)$ is consistent. Claim II shows that the  system $(E)$ has a unique solution for any $\Theta$   in the set of regular values of $co\mathscr{A}(\mathscr{P})$, which proves the first  and the second   statements of the proposition. 
The third statement comes from the fact that the set of $\Theta = (\theta_1,\ldots ,\theta_k,\psi_1 ,\ldots , \psi_k)$ for which the determinant of $\widetilde{B}(\Theta )$ vanishes is a hypersurface in the real torus and then its $2k$-volume  is zero. In other words, the  union
$
\bigcup_{s\in \mathbb{Z}_2^{2k}}Reg (co\mathscr{A}(\mathscr{P}_s))
$
is an open dense subset of the real torus $(S^1)^{2k}$. 

\hspace{12.3cm} $\Box$

\begin{corollary}\label{volume coameba}
The volume of the coamoeba of any generic $k$-dimensional linear space in $(\mathbb{C}^*)^{2k}$ is equal to $\pi^{2k}$.
\end{corollary}

\begin{proof}
By Proposition \ref{main prop} (iii), the volume of the disjoint  union 
\begin{displaymath}
\bigcup_{s\in \mathbb{Z}_2^{2k}}Reg (co\mathscr{A}(\mathscr{P}_s))
\end{displaymath}
is equal to the volume of 
all the real torus $(S^1)^{2k}$. Moreover, they have the same volume, because they are obtained from each other by translation (i.e., isometry of the real torus equipped with the flat metric). 
So, the volume of one of them must be equal to $(2\pi )^{2k}/ 2^{2k} = \pi^{2k}$.
\end{proof}

\noindent  We compute the cardinality of the inverse image under the logarithmic map of any regular value in the amoeba of a generic $k$-dimensional real very affine  linear space $\mathscr{P}(k)\subset (\mathbb{C}^*)^{2k}$.

\begin{proposition}\label{Pro NP1}
 Let $\mathscr{P}$ be a generic  real affine $k$-dimentional  linear subspace of $(\mathbb{C}^*)^{2k}$, and $x$ be a regular value 
of its amoeba. Then, the cardinality of $\Log^{-1}(x)$ is equal $2^k$.
\end{proposition}

\begin{proof}
  We  assume that $\mathscr{P}$ is given by a  parametrization $\rho$ as in $(4)$,  where all the coefficients are 
real numbers. 
 The matrix $A$ defined by:
\begin{displaymath}
{A} = \left( 
\begin{array}{ccc} 
a_{11} &\ldots& a_{1k}  \\
 \vdots & \ddots & \vdots \\
a_{k1} &\ldots& a_{kk}  
 \end{array} \right)
\end{displaymath}
is invertible, otherwise the image of $\rho$ is a linear space of dimension strictly less than $k$. The following diagram is
 commutative:
\begin{displaymath}
    \xymatrix{
(\mathbb{C}^*)^k\ar[r]^{\rho}\ar[d]_{A}&(\mathbb{C}^*)^{2k}\ar[d]^{A\times Id} \\
( \mathbb{C}^*)^k\ar[r]^{\rho'}&(\mathbb{C}^*)^{2k},    }
\end{displaymath}
where $\rho'$ is the parametrization given by:
\[
\begin{array}{ccccl}
\rho'&:&(\mathbb{C}^*)^k&\longrightarrow&(\mathbb{C}^*)^{2k}\\
&&(T_1,\ldots ,T_k)&\longmapsto&(T_1,\ldots ,T_k, b_1+T_1, \ldots , b_k+T_k).
\end{array}
\]
Each regular value of  the amoeba of the $k$-dimensional linear space $\mathscr{L} := \rho'(( \mathbb{C}^*)^k)$  is 
covered $2^k$ times under the logarithmic mapping. Indeed,   $\mathscr{L}$  is a product of  lines $L_1, \ldots ,L_k$ in $\mathbb{C}^2$.  The matrix $A$ is real,
 so the image of the set of  critical points of the logarithmic mapping restricted to $\mathscr{P}$ is the set of critical points of
 the logarithmic mapping restricted to $\mathscr{L}$. By Lemma \ref{Lem MN1}, if $z$ is a critical  in $\mathscr{P}$, then the tangent space to 
$\mathscr{L}og (\mathscr{P})$ at $\mathscr{L}og (z)$ contains at least one  purely imaginary vector $v$. Since $A$ is real, then the image of $v$ in the tangent space to $\mathscr{L}og  (\mathscr{L})$ at 
$\mathscr{L}og ((A\times Id) (z)$ is also purely imaginary tangent vector, and then,  the point $(A\times Id)(z)$ is critical.  Let $\Critp(\Log_{| \mathscr{P}})$ and $\Critp(\Log_{| \mathscr{L}})$
 be the set of critical points of the restriction of the logarithmic map to $\mathscr{P}$ and  $\mathscr{L}$ respectively. Since the volume of their amoebas is finite (see \cite{MN1-11}), 
this means that the set of critical values in their amoebas contains a subset of  dimension $2k-1$ (at least the topological boundary of the amoeba). 
Hence, the number of connected components of $\mathscr{P}\setminus \Critp(\Log_{| \mathscr{P}})$ is equal to the  number of connected 
components of $\mathscr{L}\setminus \Critp(\Log_{| \mathscr{L}})$. 
The fact
 that the set of critical points of the argument and the logarithmic maps coincide  (see   e.g., \cite{MN2-11}), 
and by Proposition \ref{main prop}, the restriction of the  
argument map to the set of  regular points  is injective,  then, the cardinality of $\Log^{-1}(x)$ is at most $2^k$. Since 
   $\mathscr{P}$ 
is real, then by Proposition \ref{Pro MN1},  for any regular value $x\in \mathscr{A}(\mathscr{P}(k))$, 
the cardinality of $\Log^{-1}(x)$ is at least $2^k$. Hence,   the cardinality of the inverse image of a regular value is equal to $2^k$. 
\end{proof}


\noindent {\it Proof of Theorem \ref{main theorem6}}.  
The first statement of Theorem \ref{main theorem6} is  Corollary \ref{volume coameba}.
The second statement of Theorem \ref{main theorem6}  is because the cardinality  of the inverse image  of a regular value in the amoeba is constant and equal to $2^k$, and the map $\Log\circ\Arg^{-1}$ conserve the volume (see \cite{MN2-11}, Proposition 3,1).  Hence, the volume of the amoeba in this case is equal to the volume of the coamoeba divided  by $2^k$.


\end{document}